\documentclass[openany, amssymb, psamsfonts]{amsart}
\usepackage{authblk}  
\usepackage{mathrsfs,comment}
\usepackage[usenames,dvipsnames]{color}
\usepackage[normalem]{ulem}
\usepackage{url}
\usepackage[all,arc,2cell]{xy}
\UseAllTwocells
\usepackage{enumerate}
\usepackage{hyperref}  
\hypersetup{%
  bookmarksnumbered=true,%
  bookmarks=true,%
  colorlinks=true,%
  linkcolor=blue,%
  citecolor=blue,%
  filecolor=blue,%
  menucolor=blue,%
  pagecolor=blue,%
  urlcolor=blue,%
  pdfnewwindow=true,%
  pdfstartview=FitBH}

\def\makeautorefname#1#2{\expandafter\def\csname#1autorefname\endcsname{#2}}
\def\equationautorefname~#1\null{(#1)\null}
\makeautorefname{footnote}{footnote}%
\makeautorefname{item}{item}%
\makeautorefname{figure}{Figure}%
\makeautorefname{table}{Table}%
\makeautorefname{part}{Part}%
\makeautorefname{appendix}{Appendix}%
\makeautorefname{chapter}{Chapter}%
\makeautorefname{section}{Section}%
\makeautorefname{subsection}{Section}%
\makeautorefname{subsubsection}{Section}%
\makeautorefname{theorem}{Theorem}%
\makeautorefname{thm}{Theorem}%
\makeautorefname{cor}{Corollary}%
\makeautorefname{lem}{Lemma}%
\makeautorefname{prop}{Proposition}%
\makeautorefname{pro}{Property}
\makeautorefname{conj}{Conjecture}%
\makeautorefname{defn}{Definition}%
\makeautorefname{notn}{Notation}
\makeautorefname{notns}{Notations}
\makeautorefname{rem}{Remark}%
\makeautorefname{quest}{Question}%
\makeautorefname{exmp}{Example}%
\makeautorefname{ax}{Axiom}%
\makeautorefname{claim}{Claim}%
\makeautorefname{ass}{Assumption}%
\makeautorefname{asss}{Assumptions}%
\makeautorefname{con}{Construction}%
\makeautorefname{prob}{Problem}%
\makeautorefname{warn}{Warning}%
\makeautorefname{obs}{Observation}%
\makeautorefname{conv}{Convention}%

\newtheorem{thm}{Theorem}[section]
\newtheorem{cor}{Corollary}[section]

\newtheorem{lem}{Lemma}[section]

\theoremstyle{definition}
\newtheorem{defn}{Definition}[section]

\makeatletter
\let\c@obs=\c@thm
\let\c@cor=\c@thm
\let\c@prop=\c@thm
\let\c@lem=\c@thm
\let\c@prob=\c@thm
\let\c@con=\c@thm
\let\c@conj=\c@thm
\let\c@defn=\c@thm
\let\c@notn=\c@thm
\let\c@notns=\c@thm
\let\c@exmp=\c@thm
\let\c@ax=\c@thm
\let\c@pro=\c@thm
\let\c@ass=\c@thm
\let\c@warn=\c@thm
\let\c@rem=\c@thm
\let\c@sch=\c@thm
\let\c@equation\c@thm
\numberwithin{equation}{section}
\makeatother

\title{Arithmetic properties of Cantor sets involving non-diagonal forms}

\author{Haotian Zhao}
\date{March 2025}

\begin{document}

\maketitle
\begin{center}
\textbf{Haotian Zhao} \\
Institute of Mathematical Sciences, ShanghaiTech University, Shanghai, China \\
\href{mailto:zhaoht2022@shanghaitech.edu.cn}{zhaoht2022@shanghaitech.edu.cn}
\end{center}

\begin{abstract}
We show conditions on $k$ such that any number \( x \) in the interval \([0, \frac{k}{2}]\) can be represented in the form \( x_1^{a_1} x_2^{a_2} + x_3^{a_3} x_4^{a_4} + \cdots + x_{k-1}^{a_{k-1}} x_k^{a_k} \), where the exponents \( a_{2i-1} \) and \( a_{2i} \) are positive integers satisfying \( a_{2i-1} + a_{2i} = s \) for \( i = 1, 2, \dots, \frac{k}{2} \), and each \( x_i \) belongs to the generalized Cantor set. Moreover, we discuss different types of non-diagonal polynomials and clarify the optimal results in low-dimensional cases.

\end{abstract}

\newpage
\section{Introduction}  
Cantor sets have long been cornerstones of analysis and set theory, serving as canonical examples of sets that are simultaneously nowhere dense, uncountable, and of measure zero. The classical Cantor ternary set $C$ is defined as the subset of \([0, 1]\) obtained by iteratively removing the middle third of each interval. This definition inspired extensive research across various mathematical disciplines. Its self-similar structure provide a rich framework for exploring concepts in topology, real analysis, and measure theory.

The arithmetic properties of Cantor sets, particularly under addition, have produced surprising and counterintuitive results. In 1917, Steinhaus proved in \cite{Stein} that the sum of two Cantor ternary sets \( C \) satisfies $C + C =  [0, 2],$
which transforms two nowhere dense, zero-measure sets into a full interval. Here we denote $C + C$ as $\{x + y : x, y \in C\}$. This result highlights the surprising algebraic richness of Cantor sets and their potential applications in modeling additive structures within fractal systems.

Waring's problem, a classical topic in number theory, aims to determine the smallest positive integer \( k \) such that every natural number \( n \) can be expressed as the sum of at most \( k \) \( s \)-th powers of natural numbers for any integer \( s \geq 3 \). Formally, this can be written as:  

\[
n = x_1^s + x_2^s + \cdots + x_k^s, \quad x_i \in \mathbb{N}.
\]
Extending this question to Cantor sets presents novel challenges and opportunities for discovery.

In 2019, Athreya, Reznick, and Tyson studied the arithmetic properties of the Cantor set \cite{Tyson} and conjectured that every element in \([0, 1]\) could be expressed as:
\[
x_1^2 + x_2^2 + x_3^2 + x_4^2, \quad x_j \in C, \, 1 \leq j \leq 4.
\]
This conjecture was later proved by Wang, Jiang, Li, and Zhao \cite{Zhao}, who broadened the result to encompass generalized middle-\(\frac{1}{\alpha}\) Cantor sets. Let \(C_\alpha\) denote the set constructed by successively removing middle open intervals of length \(\frac{1}{\alpha}\) at each stage, where \(\alpha > 1\). They established that for all $\alpha \geq 3$, the following identity holds:

\begin{equation}\label{1.1}
[0, 4] = \{x_1^2 + x_2^2 + x_3^2 + x_4^2 : x_j \in C_{\alpha}, \, 1 \leq j \leq 4\}.
\end{equation}

Further advancements in the Waring–Hilbert problem on Cantor sets were made by Yinan Guo \cite{Guo} in 2021. Guo highlighted a key result by systematically providing comprehensive conclusions for the Waring problem on Cantor sets. He demonstrated that for any integer \(s \geq 3\), there exists \(k \leq 6^s\) such that every real number in \([0, 1]\) can be expressed as the sum of \(k\) \(s\)-th powers of elements in the Cantor set \(C\). Formally:
\begin{equation}\label{1.2}
[0, 1] \subseteq \left\{x_1^s + x_2^s + \cdots + x_k^s : x_j \in C, \; 1 \leq j \leq k \right\}.
\end{equation}

Beyond his foundational work on the classical Cantor set, Guo extended his research to Cantor dust, which is the Cartesian product \( C \times C \) of the Cantor set with itself. This higher-dimensional generalization enabled a deeper exploration of additive structures in fractal systems. Specifically, when \( C \times C \) is embedded in the complex plane \(\mathbb{C}\), Guo studied the Waring–Hilbert problem in this setting and showed that it has a positive solution for power sums up to degree 4. His results provide valuable insights into the relationship between number theory and fractal geometry, particularly in the study of diagonal form representations within higher-dimensional fractal structures.

Based on Guo's result (\ref{1.2}), Lu Cui and Minghui Ma introduced a refined framework, inspired by the approach in (\ref{1.1}), to further strengthen Guo's conclusions presented in \cite{Cui}. Their work introduced significant advancements by extending the analysis to the generalized middle-\(\frac{1}{\alpha}\) Cantor set \(C_\alpha\). Cui and Ma demonstrated that for any \(\alpha \in (1, 2 + \sqrt{5})\) and sufficiently large \(s > 1\), the smallest positive integer \(k\) satisfying the following relation is given by 
\(
k = \left\lceil\left(\frac{1}{r} - 1\right)^s\right\rceil,
\)
where \(r = \frac{1}{2}\left(1 - \frac{1}{\alpha}\right)\), and:

\begin{equation}\label{4}
[0, k] = \left\{x_1^s + x_2^s + \cdots + x_k^s : x_j \in C_\alpha, \, 1 \leq j \leq k\right\}.
\end{equation}

Notably, the results of \cite{Cui} are optimal for \(\alpha \in (1, 2 + \sqrt{5})\) and  \(s > 1\), representing a significant advancement in the study of Cantor-type sets in number theory.

In addition to their extensive study of diagonal forms, Jayadev S. Athreya, Bruce Reznick, and Jeremy T. Tyson also explored non-diagonal forms in lower dimensions \cite{Tyson}. They demonstrated that every element \( u \in [0,1] \) can be expressed as \( u = x^2 y \), where \( x, y \) are elements of the Cantor set \( C \). However, they also showed that the set of real numbers \( v \) representable in the form \( v = xy \) with \( x, y \in C \) forms a closed subset of \([0,1]\) with Lebesgue measure strictly between \( \frac{17}{21} \) and \( \frac{8}{9} \).

The main goal of this paper is to study the behavior of generalized, non-diagonal forms defined on Cantor sets, with particular emphasis on their value ranges. Specifically, we focus on polynomials \(P_k\) of the form:

\[
P_k(x_1, x_2, \ldots, x_k) = x_1^{a_1} x_2^{a_2} + x_3^{a_3} x_4^{a_4} + \cdots + x_{k-1}^{a_{k-1}} x_k^{a_k},
\]
where  \(a_{2i-1}, a_{2i} > 0\) are integers satisfying \(a_{2i-1} + a_{2i} = s\) for \(i = 1, 2, \dots, \frac{k}{2}\). Here, \(k\) is a positive even integer, \(s \geq 2\) is a given constant, and \(C_\alpha\) denotes the generalized Cantor set.

To avoid excessive notation, the generalized case for Cantor sets is deferred to later sections. For a detailed discussion, see Theorems \ref{FMT} and \ref{TMT}.
Here, for clarity, we present the theorem for the specific case of the Cantor ternary set \(C\), which highlights the key structure and results.

\begin{thm}\label{MResult}  
Let \(a_i > 0\) (\(i = 1, 2, \ldots, k\)) be integers satisfying:  
\[
a_{2i-1} + a_{2i} = s, \quad i = 1, 2, \ldots, \frac{k}{2},
\]   
where \(s > 2\) is a given constant. Then, for any positive even integer \(k\) satisfying:  
\[
k \geq 2^{\frac{s}{2}+1} \cdot \left\lceil \left(\frac{s+3}{s-2}\right)^{\frac{s}{2} - 1} \left( \frac{5 s + 6 }{5 s - 6 }\right) + 1 \right\rceil,
\]
the following holds:  
\[
[0, \frac{k}{2}] = \left\{ \sum_{i=1}^{k/2} x_{2i-1}^{a_{2i-1}} x_{2i}^{a_{2i}} \;\middle|\; x_1, x_2, \dots, x_k \in C \right\}.
\]
\end{thm}  

The outline of this paper is as follows:  

In Section \ref{sec2}, we introduce the necessary notation and establish several auxiliary results. For foundational concepts in number theory, set theory, and real analysis, the reader is referred to \cite{Hua1982, Pan1991, Rudin1987, Ye2008}.  

In Section \ref{sec3}, we adapt the theorem to the generalized Cantor set \(C_\alpha\), building on the framework established in previous works. We first consider the generalization of the theorem for given values of \(a\) and \(b\). The value range of the function \(f_{k,a,b}\) on \(C^k_\alpha\) forms a compact interval \([0, \frac{k}{2}]\) for sufficiently large \(k\), with the precise value of \(k\) determined by Theorem \ref{FMT}. We also derive Corollary \ref{SMT}, which addresses the case where \(s = a + b\) is given, but \(a\) and \(b\) are allowed to vary. This underscores the generality of the result and demonstrates that the chosen \(k\) is sufficient to account for all valid configurations of \(a\) and \(b\) under the constraint \(a + b = s\).  

In Section \ref{sec4}, we refine the theorem by fixing \(s\) and allowing arbitrary choices of \(a_i\) under the condition \(a_{2i-1} + a_{2i} = s\), which culminates in Theorem \ref{TMT}. It is evident that Theorem \ref{MResult}, introduced in the introduction, naturally follows as a corollary of Theorem \ref{TMT}.

In Section \ref{sec5}, we address cases in lower dimensions and provide further discussions. In low-dimensional cases, such as \(s = 2\) or \(s = 3\), the minimal \(k\) values are \(k = 4\) and \(k = 2\), respectively. For more general forms like \(x^a y^b\), it is shown in \cite{Tyson} that \(b/a \geq \frac{\log 2}{\log \frac{3}{2}}\) ensures full coverage of \([0, 1]\), while smaller ratios introduce gaps. Our theorem provides a unified framework to address all such cases. 

\section{Notation and Fundamental Lemma}\label{sec2}

We denote \(\mathbb{Q}, \mathbb{Z}, \mathbb{N}, \mathbb{R},\) and \(\mathbb{C}\) as the sets of rational numbers, integers, non-negative integers, real numbers, and complex numbers, respectively. Throughout this article, we assume that $a, b, i, j,$ and $s$ are integers, and that $k$ is a even integer.

\begin{defn}[Operations on Sets]
For non-empty subsets \(X, Y \subseteq \mathbb{R}\) or \(\mathbb{C}\), and for \(t \in \mathbb{R}\) or \(\mathbb{C}\), we define:
\[
X + Y = \{x + y : x \in X, y \in Y\}, \quad XY = \{xy : x \in X, y \in Y\},
\]
\[
t + X = \{t + x : x \in X\}, \quad tX = \{tx : x \in X\},
\]
\[
X^k = \{(x_1, x_2, \ldots, x_k) : x_1, x_2, \ldots, x_k \in X\}.
\]
\end{defn}
\begin{defn}[Binary Representations]
Let \(\mathbb{Z}_2^k = \{0, 1\}^k\) denote the set of all binary strings of length \(k\), where each element is written as \(\sigma_1\sigma_2\cdots\sigma_k\) with \(\sigma_i \in \{0, 1\}\). If \(\omega_1 = \sigma_1\sigma_2\cdots\sigma_{k_1} \in \mathbb{Z}_2^{k_1}\) and \(\omega_2 = \sigma_1'\sigma_2'\cdots\sigma_{k_2}' \in \mathbb{Z}_2^{k_2}\), their operation is defined as  
\[
\omega_1 \omega_2 = \sigma_1\sigma_2\cdots\sigma_{k_1}\sigma_1'\sigma_2'\cdots\sigma_{k_2}' \in \mathbb{Z}_2^{k_1 + k_2}.
\]
\end{defn}
\begin{defn}
We define the subset \(S_2^k \subset \mathbb{Z}_2^k\), consisting of all elements \(\omega \in \mathbb{Z}_2^k\) that satisfy the number of \(1\)s in \(\omega\) is even.
Explicitly, the elements of \(S_2^k\) are given by  
\[
(0, 0, \ldots, 0), \quad (0, 0, \ldots, 0, 1, 1), \quad \ldots,\quad(1, 1, \ldots, 1, 1)
\]  
where for each \(\omega_i \in S_2^k\), the first \(k - 2(i - 1)\) digits are \(0\) and the remaining digits are \(1\).
\end{defn}
A total order \(0 < 1\) is defined on \(\mathbb{Z}_2\), which naturally induces a dictionary order on \(S_2^k\). For any two adjacent elements \(\omega_i, \omega_{i+1} \in S_k^2\), we define  
\[
\omega_{i+1} = \omega_i \oplus 1,
\]  
where $\oplus$ denotes the bitwise exclusive OR operation, flipping the least significant bit of $\omega_i$.
This operation ensures that each successive element differs from the previous one by exactly one bit, following the standard Gray code pattern.

Recall from the Introduction, we now provide a detailed definition of the Cantor set.
\begin{defn}[Cantor Set]
For the Cantor set \(C_\alpha\), let \(\alpha > 1\) be a real number and define \(r = \frac{1}{2}(1 - \frac{1}{\alpha}) \in (0, \frac{1}{2})\). (When \(\alpha = 3\), the Cantor ternary set \(C\) is obtained.) Define two linear (and increasing) functions \(f_0, f_1 : [0, 1] \to [0, 1]\) as follows:  
\[
f_0(x) = rx, \quad f_1(x) = rx + (1 - r).
\]

The Cantor set \(C_\alpha\) satisfies:  
\[
C_\alpha = f_0(C_\alpha) \cup f_1(C_\alpha) = rC_\alpha \cup (rC_\alpha + 1 - r).
\]  

For \(n \geq 1\) and \(\omega = \sigma_1\sigma_2\cdots\sigma_n \in \mathbb{Z}_2^{n}\), we define \(f_\omega : [0, 1] \to [0, 1]\) by  
\[
f_\omega(x) = f_{\sigma_1} \circ f_{\sigma_2} \circ \cdots \circ f_{\sigma_n}(x),
\]  
where \(\circ\) denotes the composition of functions \(f_{\sigma_1}, \ldots, f_{\sigma_n}\) for \(\sigma_1, \ldots, \sigma_n \in \mathbb{Z}_2\). It can be verified that \(f_\omega([0, 1])\) is a closed interval contained in \([0, 1]\) with length \(r^n\). 
\end{defn}
The following result, derived from Guo's work\cite{Guo}, provides an explicit formula for the function \( f_\omega(x) \) and the structure of the interval \( f_\omega([0, 1]) \):
\begin{lem}
Let \( \omega = \sigma_1\sigma_2\cdots\sigma_n \in \mathbb{Z}_2^{n} \), where \(\mathbb{Z}_2^{n}\) denotes the set of binary strings of length \(n\), and let \(x \in [0, 1]\). Define a map \(v: \mathbb{Z}_2 \to \{1, 2\}\) as \(v(0) = 1\) and \(v(1) = 2\). Then, the function \(f_\omega: [0, 1] \to [0, 1]\) satisfies:  
\[
f_\omega(x) = \frac{1 - r}{r} \sum_{k=1}^n (v(\sigma_k) - 1) r^k + x r^n,
\]  
where \(r = \frac{1}{2}(1 - \frac{1}{\alpha})\) for a given \(\alpha > 1\).

Moreover, the interval \(f_\omega([0, 1])\) is given explicitly as:  
\[
f_\omega([0, 1]) = \left[\sum_{k=1}^n \frac{2(v(\sigma_k) - 1)}{3^k}, \sum_{k=1}^n \frac{2(v(\sigma_k) - 1)}{3^k} + \frac{1}{3^n}\right].
\]
\end{lem}
Such an interval \( f_\omega([0, 1]) \) is called a level \( n \) basic interval. In fact, the order in \( \mathbb{Z}_2^{n} \) gives rise to a total order of level \( n \) basic intervals as intervals of \([0, 1]\) under the usual order of real numbers, i.e., for any \( \omega_1, \omega_2 \in \mathbb{Z}_2^{n} \), whenever \( \omega_1 < \omega_2 \), we have \( f_{\omega_1}(x_1) < f_{\omega_2}(x_2) \) for any \( x_1, x_2 \in [0, 1] \).

To clarify and avoid potential confusion, we provide the following definitions.
\begin{defn}\label{defn}
For \(n \geq 1\), define:  
\[
F_n = \{f_\omega([0, 1]) : \omega \in \{0, 1\}^n\},
\]
and
\[
C_n = \bigcup_{I \in F_n} I \subseteq [0, 1].
\]

Each set \(C_n\) is a union of finitely many closed intervals, and it satisfies \(C_n \supseteq C_{n+1}\). Furthermore, the Cantor set \(C_\alpha\) can be expressed as:  
\[
C_\alpha = \bigcap_{n=1}^\infty C_n = \bigcap_{n=1}^\infty \left(\bigcup_{\omega \in \mathbb{Z}_2^{n}} f_\omega([0, 1])\right).
\]
\end{defn}
Notice that each level  $n$ basic interval is uniquely determined by its left endpoint and its level. The left endpoints of these intervals play a crucial role in organizing the structure of the Cantor set. To formalize this idea, we introduce the following notation for the left endpoints and intervals at each level.
\begin{defn}
We denote \(L_n = \{f_\omega(0) : \omega \in \mathbb{Z}_2^{n}\}\) as the set of all left endpoints of intervals in \(F_n\). For any \(u = f_\omega(0) \in L_n\), we define \(I_u(n)\) to represent the interval \(f_\omega([0, 1]) = [u, u + r^n]\). When no ambiguity arises, we simplify the notation to:  
\begin{equation}\label{leftp1}
I_u = [u, u + r^n] \in F_n.
\end{equation}

This definition is motivated by the observation that each level \(n\) basic interval is uniquely determined by its left endpoint and its level.  

Additionally, any level \(n\) basic interval \(f_\omega([0, 1])\) contains two level \(n+1\) basic intervals, namely \(f_{\omega 0}([0, 1])\) and \(f_{\omega 1}([0, 1])\), which belong to \(F_{n+1}\). We further define:  
\[
I_{u,0} = I_u(n+1) = [u, u + r^{n+1}],
\]
\[
I_{u,1} = I_{u + 2r^{n+1}}(n+1) = [u + 2r^{n+1}, u + r^n].
\]

For a unified representation, for \(\sigma \in \{0, 1\}\), we write:  
\begin{equation}\label{leftp2}
I_{u,\sigma} = [u + 2 r^{n+1}\sigma, u + r^{n+1} + 2 r^{n+1}\sigma].
\end{equation}
\end{defn}
The following lemma is a standard result in real analysis \cite{Rudin1987} and has been referenced by several authors . Let \(\mathbb{R}^k\) denote the \(k\)-dimensional Euclidean space.
\begin{lem}
Let \(f : \mathbb{R}^k \to \mathbb{R}\) be a real-valued continuous function, where \(k \geq 1\). If \(\{X_n\}_{n=1}^\infty\) is a decreasing sequence of non-empty compact subsets of \(\mathbb{R}^k\), then the following holds:  
\[
f\left( \bigcap_{n=1}^\infty X_n \right) = \bigcap_{n=1}^\infty f(X_n).
\]
\end{lem}

We define \(s = a + b\). For any \(s \geq 2\) and positive even integer k, and any \((x_1, x_2, \ldots, x_k) \in \mathbb{R}^k\), the function \(f_{k,a,b}\) is given by:  
\[
f_{k,a,b}(x_1, x_2, \ldots, x_k) = x_1^a x_2^b + x_3^a x_4^b + \cdots + x_{k-1}^a x_k^b.
\]

In a more general setting, we define the function \(f_{k,\vec{a}}\) as:  
\[
f_{k,\vec{a}}(x_1, x_2, \ldots, x_k) = \sum_{i=1}^{k/2} x_{2i-1}^{a_{2i-1}} x_{2i}^{a_{2i}},
\]
where \(x_1, x_2, \ldots, x_k \in C_\alpha\), and \(\vec{a} = (a_1, a_2, \ldots, a_k)\) is a \(k\)-dimensional vector of positive integers. Furthermore, the integer components \(a_{2i-1}\) and \(a_{2i}\) of \(\vec{a}\) satisfy the condition \(a_{2i-1} + a_{2i} = s\) for each \(i = 1, 2, \ldots, k/2\).

This definition extends the original \(f_{k,a,b}\) by allowing the exponents \(a_{2i-1}\) and \(a_{2i}\) to vary within each term, provided their sum is fixed as \(s\). This generalization accommodates greater flexibility in the construction of the function while maintaining its structural consistency.  

For \(C_n = \bigcup_{I \in F_n} I \subseteq [0, 1]\) defined earlier, \(C_n^k\) is a compact subset of \(\mathbb{R}^k\) and forms a decreasing sequence in \(n\). It follows that:  
\[
C_\alpha^k = \bigcap_{n=1}^\infty C_n^k.
\]

Thus, we derive the following corollary:
\begin{cor}\label{linear}  
Let \(f_{k,a,b}\) and \(f_{k,\vec{a}}\) be real-valued continuous functions defined on \(\mathbb{R}^k\). For \(s \geq 1\), \(k \geq 1\), and \(k \equiv 0 \pmod{2}\), the following holds:  
\[
f_{k,a,b}(C_\alpha^k) = f_{k,a,b}\left(\bigcap_{n=1}^\infty C_n^k\right) = \bigcap_{n=1}^\infty f_{k,a,b}(C_n^k),
\]
and similarly,  
\[
f_{k,\vec{a}}(C_\alpha^k) = f_{k,\vec{a}}\left(\bigcap_{n=1}^\infty C_n^k\right) = \bigcap_{n=1}^\infty f_{k,\vec{a}}(C_n^k).
\]
\end{cor}  

The main goal of this article is to analyze the behavior of both \(f_{k,a,b}\) and \(f_{k,\vec{a}}\) on \(C_\alpha^k\). In particular, we provide the following fundamental results to characterize their respective value sets and structural properties when restricted to the Cantor set \(C_\alpha^k\).  
\begin{lem}\label{linearcomb}
For \(s \geq 1\), \(\alpha > 1\), and \(k \in \mathbb{N}\), if \(X \subseteq f_{k,a,b}(C_\alpha^k)\), then \(r^{(l_1a + l_2b)}X \subseteq f_{k,a,b}(C_\alpha^k)\) for any \(l_1, l_2 \in \mathbb{N}\).
\end{lem}

\begin{proof}
Let \(x \in X \subseteq f_{k,a,b}(C_\alpha^k)\). Then,  
\[
x = f_{k,a,b}(x_1, x_2, \ldots, x_k),
\]  
for some \((x_1, x_2, \ldots, x_k) \in C_\alpha^k\). Since \(x_i \in C_\alpha\) implies \(r x_i \in C_\alpha\) for \(i = 1, 2, \ldots, k\), we have \((r x_1, x_2, r x_3, x_4, \ldots, r x_{k-1}, x_k) \text{ and }  (x_1, r x_2, x_3, r x_4, \ldots, x_{k-1}, r x_k) \in C_\alpha^k\). By the definition of \(f_{k,a,b}\), we obtain:  
\[
r^a x = r^a f_{k,a,b}(x_1, x_2, \ldots, x_k) = f_{k,a,b}(r x_1, x_2, r x_3, x_4, \ldots, r x_{k-1}, x_k) \in f_{k,a,b}(C_\alpha^k).
\]
\[
r^b x = r^b f_{k,a,b}(x_1, x_2, \ldots, x_k) = f_{k,a,b}(x_1, r x_2, x_3, r x_4, \ldots, x_{k-1}, r x_k) \in f_{k,a,b}(C_\alpha^k).
\]

By combining these two equations through linear combinations with coefficients \(l_1, l_2 \in \mathbb{N}\), we deduce:  
\[
r^{(l_1a + l_2b)} x \in f_{k,a,b}(C_\alpha^k).
\]  
Thus, \(r^{(l_1a + l_2b)} X \subseteq f_{k,a,b}(C_\alpha^k)\), completing the proof. 
\end{proof}

\begin{lem}\label{linearcomb-fks}
For \(s \geq 1\), \(\alpha > 1\), and \(k \in \mathbb{N}\), if \(X \subseteq f_{k,\vec{a}}(C_\alpha^k)\), then \(r^{ls} X \subseteq f_{k,\vec{a}}(C_\alpha^k)\) for any \(l \in \mathbb{N}\).  
\end{lem}

\begin{proof}
Let \(x \in X \subseteq f_{k,\vec{a}}(C_\alpha^k)\). Then,  
\[
x = f_{k,\vec{a}}(x_1, x_2, \ldots, x_k),
\]  
for some \((x_1, x_2, \ldots, x_k) \in C_\alpha^k\). Since \(x_i \in C_\alpha\) implies \(r x_i \in C_\alpha\) for all \(i = 1, 2, \ldots, k\), we have \((r x_1, r x_2, \ldots, r x_k) \in C_\alpha^k\). By the definition of \(f_{k,\vec{a}}\),  
\[
r^s x = r^s f_{k,\vec{a}}(x_1, x_2, \ldots, x_k) = f_{k,\vec{a}}(r x_1, r x_2, \ldots, r x_k).
\]  
Since \((r x_1, r x_2, \ldots, r x_k) \in C_\alpha^k\), it follows that  
\[
r^s x \in f_{k,\vec{a}}(C_\alpha^k).
\]

Thus, \(r^{ls} X \subseteq f_{k,\vec{a}}(C_\alpha^k)\), completing the proof.  
\end{proof}

We conclude this section with two lemmas that establish the relationship between \(f_{k,a,b}(C^{k}_\alpha)\), \(f_{k+2,a,b}(C^{k+2}_\alpha)\), and their generalized form \(f_{k,\vec{a}}(C^{k}_\alpha)\), \(f_{k+2,\vec{a}}(C^{k+2}_\alpha)\). These results demonstrate how these functions extend naturally with the addition of new variables.  

The proofs follow directly from the definitions of \(f_{k,a,b}\) and \(f_{k,\vec{a}}\), highlighting their additive structures.  
\begin{lem}\label{plus}  
For \(s, k, \alpha \geq 1\), if \(X \subseteq f_{k,a,b}(C_\alpha^k)\), then for any \(x_{k+1}, x_{k+2} \in C_\alpha\):  
\[
X + x_{k+1}^a x_{k+2}^b \subseteq f_{k+2,a,b}(C_\alpha^{k+2}).
\]  
In particular,  
\[
f_{k,a,b}(C_\alpha^k) \subseteq f_{k+2,a,b}(C_\alpha^{k+2}).
\]
\end{lem}

\begin{lem}\label{plus_fks}  
For \(s, k, \alpha \geq 1\), if \(X \subseteq f_{k,\vec{a}}(C_\alpha^k)\), then for any \(x_{k+1}, x_{k+2} \in C_\alpha\) and \(a_{k+1}, a_{k+2} > 0\) such that \(a_{k+1} + a_{k+2} = s\):  
\[
X + x_{k+1}^{a_{k+1}} x_{k+2}^{a_{k+2}} \subseteq f_{k+2,\vec{a}}(C_\alpha^{k+2}).
\]  
In particular,  
\[
f_{k,\vec{a}}(C_\alpha^k) \subseteq f_{k+2,\vec{a}}(C_\alpha^{k+2}).
\]
\end{lem}

\section{Proof of the Fundamental Theorem}\label{sec3}

In this section, we first establish Theorem \ref{FMT}, which serves as the foundational step toward proving Theorem \ref{TMT}. Formally, Theorem \ref{FMT} can be viewed as a simplified case of Theorem \ref{TMT}, where all \(a_{2i-1}\) are fixed as \(a\) and all \(a_{2i}\) are fixed as \(b\). This result lays the groundwork for extending the proof to more general cases.

Before stating the theorem, we recall some relevant notation here. The set \( C_\alpha \) is constructed by repeatedly removing middle open intervals of length \( \frac{1}{\alpha} \), where \( \alpha > 1 \). Additionally, we define \(
r = \frac{1}{2} \left(1 - \frac{1}{\alpha} \right),
\) which will play a crucial role in the subsequent discussion.

\begin{thm}\label{FMT}  Suppose that $\alpha > 1$,
let \(a, b \in \mathbb{Z^+}\) and \(s = a + b\) be an integer with \(s \geq 2\). Assume \(a \leq b\), and let  
\[
n_* = n_*(r, a) = \lfloor -\log_r a \rfloor + 1.
\]  
Then, for any positive even integer \(k\) satisfying  
\[
k \geq \max \Bigg\{
\left\lceil2\frac{(1-r+r^{n_*})^{a-1}[b(1-r+r^{n_*})+a]}{(1-r)^{a-1}(1-r^{n_*})^{b-1}[b(1-r)+a(1-r^{n_*})]} + 2\right\rceil \cdot \left(\frac{1-r}{r}\right)^a,
\]
\[
2\left(\frac{1-r}{r}\right)^{a+1},
2\left(\frac{1}{1-r}\right)^{s-1} + 2
\Bigg\},
\]  
the following holds:
\[
[0, \frac{k}{2}] = \left\{x_1^a x_2^b + x_3^a x_4^b + \cdots + x_{k-1}^a x_k^b : x_1, x_2, \ldots, x_k \in C_\alpha \right\}.
\]
\end{thm}  

\subsection{Preparatory Lemmas}
Before presenting the proof of this theorem, some preliminary results are necessary. We will rely on the notations and results established in Section \ref{sec2}. The core idea of the proof involves demonstrating that the value range of a given function acting on the discrete point set of the Cantor set is compact. Specifically, we aim to show that the value range can form a connected interval as stated in the theorem. To illustrate this, we begin with the following lemmas, which establish the interval property. A complete proof of the theorem will be provided at the end of this section.

In the subsequent two lemmas, we employ the Multivariate Differential Mean Value Theorem to generalize Lemmas 3.1 and 3.2 from \cite{Cui} to the function \(f_{k,a,b}(C_{\alpha}^k)\). These new lemmas provide a novel criterion for identifying intervals within \(f_{k,a,b}(C_{\alpha}^k)\).
\begin{lem}\label{fundlem}
Suppose \(k \in \mathbb{N}\), \(k \geq 2\), and \(u_1, u_2, \ldots, u_k \in L_n\). If the following inequality holds:  
\[
\sum_{i \neq M, i \equiv 1 \pmod{2}} u_i^{a-1} u_{i+1}^{b-1} (b u_i + a u_{i+1}) \geq (u_M + r^n - r^{n+1})^{a-1} (u_{M+1} + r^n - r^{n+1})^{b-1} 
\]
\[
\cdot \, [b (u_M + r^n - r^{n+1}) + a (u_{M+1} + r^n - r^{n+1})],
\]  
where \(u_M\) and \(u_{M+1}\) satisfy  
\[
u_M^a u_{M+1}^b = \max\{u_1^a u_2^b, u_3^a u_4^b, \ldots, u_{k-1}^a u_k^b\},
\]  
then the following equality holds:  
\[
f_{k,a,b}(I_{u_1} \times I_{u_2} \times \cdots \times I_{u_k}) = f_{k,a,b}[(I_{u_1,0} \cup I_{u_1,1}) \times (I_{u_2,0} \cup I_{u_2,1}) \times \cdots \times (I_{u_k,0} \cup I_{u_k,1})],
\]  
where \(I_{u_j}\) and \(I_{u_j,\sigma}\) are defined as in Equations (\ref{leftp1}) and (\ref{leftp2}), respectively.
\end{lem}

\begin{proof}
It is clear that the right-hand side is a subset of the left-hand side since \(I_{u_i,0} \cup I_{u_i,1} \subseteq I_{u_i}\) for \(i = 1, 2, \ldots, k\). Hence, we only need to prove the reverse inclusion:  
\[
f_{k,a,b}(I_{u_1} \times I_{u_2} \times \cdots \times I_{u_k}) \subseteq f_{k,a,b}((I_{u_1,0} \cup I_{u_1,1}) \times (I_{u_2,0} \cup I_{u_2,1}) \times \cdots \times (I_{u_k,0} \cup I_{u_k,1})).
\]

Note that for any \(u_1', u_2', \ldots, u_k' \in L_{n+1}\), the product \(I_{u_1'} \times \cdots \times I_{u_k'}\) is connected and compact. Consequently, \(f_{k,a,b}(I_{u_1} \times I_{u_2} \times \cdots \times I_{u_k})\) is a closed interval with endpoints:  
\[
\text{Left endpoint: } u_1^a u_2^b + \cdots + u_{k-1}^a u_k^b,
\]  
\[
\text{Right endpoint: } (u_1 + r^{n+1})^a(u_2 + r^{n+1})^b + \cdots + (u_{k-1} + r^{n+1})^a(u_k + r^{n+1})^b.
\]

Let \(Q_v = f_{k,a,b}(I_{u_1,v_1} \times I_{u_2,v_2} \times \cdots \times I_{u_k,v_k})\) for \(v = v_1 v_2 \cdots v_k \in \{0,1\}^k\). Then \(Q_v\) is an interval, and:  
\[
f_{k,a,b}((I_{u_1,0} \cup I_{u_1,1}) \times (I_{u_2,0} \cup I_{u_2,1}) \times \cdots \times (I_{u_k,0} \cup I_{u_k,1})) = \bigcup_{v \in \{0,1\}^k} Q_v.
\]

Thus, it suffices to prove that \(\bigcup_{v \in \{0,1\}^k} Q_v\) is connected. To do this, we use an analog of the Manhattan distance on \(\{0,1\}^k\): for any \(v = v_1 v_2 \cdots v_k\) and \(w = w_1 w_2 \cdots w_k \in \{0,1\}^k\), define:  
\[
d(v,w) = \frac{1}{2} \sum_{i=1}^k |v_i - w_i|.
\]

Using the total order on \(S_2^k\), we prove by induction that for any \(\omega_n \in S_2^k\), \(\bigcup_{\omega \leq \omega_n} Q_{\omega}\) is an interval. This is clearly true when \(\omega_1 = 0 \cdots 0 \in S_2^k\). Assume for some \(\omega_i \in S_2^k\), \(\bigcup_{\omega \leq \omega_i} Q_{\omega}\) is an interval. Consider the next element \(\omega_{i+1} \in S_2^k\), and we need to show that \(\bigcup_{\omega \leq \omega_{i+1}} Q_{\omega}\) is also an interval.

To do this, we must show that the left endpoint of \(Q_{\omega_{i+1}}\), denoted 
 by \(L(\omega_{i+1})\), is less than or equal to the right endpoint of \(Q_{\omega_i}\), denoted by \(R(\omega_i)\). By the definition of \(\omega_i\), we have:  
\[
\omega_i = \underbrace{00\cdots00}_{k-2(i-1)} \underbrace{11\cdots11}_{2(i-1)}, \quad 
\omega_{i+1} = \underbrace{00\cdots00}_{k-2i} \underbrace{11\cdots11}_{2i}.
\]
 
where \(\omega_i\) has \(k-2(i-1)\) zeros and \(2(i-1)\) ones, while \(\omega_{i+1}\) has \(k-2i\) zeros and \(2i\) ones. 

For \(Q_{\omega_i}\) and \(Q_{\omega_{i+1}}\), the left and right endpoints are given as follows:

The left endpoint of \(Q_{\omega_{i+1}}\):  
\[
L(\omega_{i+1}) = \sum_{l = 1}^{j-2} u_l^a u_{l+1}^b + \sum_{l = j}^{k-1} (u_l + 2r^{n+1})^a (u_{l+1} + 2r^{n+1})^b,
\]  
where the first summation runs over \(l = 1, 3, \ldots, j-2\), and the second summation runs over \(l = j, j+2, \ldots, k-1\).

The right endpoint of \(Q_{\omega_i}\):  
\[
R(\omega_i) = \sum_{l = 1}^{j} (u_l + r^{n+1})^a (u_{l+1} + r^{n+1})^b + \sum_{l = j+2}^{k-1} (u_l + r^n)^a (u_{l+1} + r^n)^b,
\]  
where the first summation runs over \(l = 1, 3, \ldots, j\), and the second summation runs over \(l = j+2, j+4, \ldots, k-1\).

The difference between \(R(\omega_i)\) and \(L(\omega_{i+1})\) can then be written as:  
\[
R(\omega_i) - L(\omega_{i+1}) = \sum_{l = 1}^{j-2} \Big[(u_l + r^{n+1})^a (u_{l+1} + r^{n+1})^b - u_l^a u_{l+1}^b\Big]
\]
\[
+ \Big[(u_j + r^{n+1})^a (u_{j+1} + r^{n+1})^b - (u_j + 2r^{n+1})^a (u_{j+1} + 2r^{n+1})^b\Big]
\]
\[
+ \sum_{l = j+2}^{k-1} \Big[(u_l + r^n)^a (u_{l+1} + r^n)^b - (u_l + 2r^{n+1})^a (u_{l+1} + 2r^{n+1})^b\Big].
\]

Using the Multivariate Differential Mean Value Theorem and simplifying the difference \(R(\omega_i) - L(\omega_{i+1})\), we aim to rewrite \(R(\omega_i)\) and \(L(\omega_{i+1})\) in a unified form, as the values of the function \(f_{k,a,b}\) acting on specific vectors. This approach facilitates the application of the theorem to derive the desired bounds.

To achieve this, we define two vectors \(\vec{x}_i\) and \(\vec{x}_{i+1}\), corresponding to \(R(\omega_i)\) and \(L(\omega_{i+1})\), respectively. Specifically:  
\[
\vec{x}_i = (x_{i1}, x_{i2}, \ldots, x_{ik}), \quad \vec{x}_{i+1} = (x_{i+1,1}, x_{i+1,2}, \ldots, x_{i+1,k}),
\]  
where:
\begin{itemize}
\item For \(\vec{x}_i\), the components are defined as:
  \[
  x_{il} =
  \begin{cases} 
  u_l + r^{n+1}, & l = 1, 2, \ldots, j, j+1 \\
  u_l + r^n, & l = j+2, j+3, \ldots, k.
  \end{cases}
  \]
\item For \(\vec{x}_{i+1}\), the components are defined as:
  \[
  x_{i+1,l} =
  \begin{cases} 
  u_l, & l = 1, 2, \ldots, j-1, \\
  u_l + 2r^{n+1}, & l = j, j+1, \ldots, k.
  \end{cases}
  \]
\end{itemize}
With these definitions, we have:
\[
R(\omega_i) = f_{k,a,b}(\vec{x}_i), \quad L(\omega_{i+1}) = f_{k,a,b}(\vec{x}_{i+1}).
\]

By the Multivariate Differential Mean Value Theorem, there exists a point \(\vec{\xi} = (\xi_1, \xi_2, \ldots, \xi_k)\) lying on the line segment connecting \(\vec{x}_i\) and \(\vec{x}_{i+1}\) such that:
\[
f_{k,a,b}(\vec{x}_i) - f_{k,a,b}(\vec{x}_{i+1}) = \sum_{l=1}^k \frac{\partial f_{k,a,b}}{\partial x_l}(\xi_l) (x_{il} - x_{i+1,l}),
\]
where \(x_{il} - x_{i+1,l}\) corresponds to the differences in their respective components, explicitly:  
\[
x_{il} - x_{i+1,l} =
\begin{cases} 
r^{n+1}, & l = 1, 2, \ldots, j-2, j-1 \\
r^{n+1} -  2r^{n+1} = -r^{n+1}, & l = j, j+1\\
r^n - 2r^{n+1} = r^{n+1}, & l = j+2, j+3, \ldots, k.
\end{cases}
\]

In this specific case, the partial derivatives are given by:
\[
\frac{\partial f_{k,a,b}}{\partial x_l} (\xi_l) =
\begin{cases} 
a \xi_l^{a-1} \xi_{l+1}^b , & l \equiv 1 \pmod{2}, \\
b \xi_{l-1}^a \xi_{l}^{b-1}, & l \equiv 0 \pmod{2}.
\end{cases}
\]

Specifically, we have:  
\[
f_{k,a,b}(\vec{x}_i) - f_{k,a,b}(\vec{x}_{i+1}) = 
r^{n+1} \Bigg[\sum_{l=1, l \equiv 1 \pmod{2}}^{j-2} \big(a\xi_l^{a-1} \xi_{l+1}^b + b\xi_l^a \xi_{l+1}^{b-1}\big)
\]
\[
- \big(a\xi_j^{a-1} \xi_{j+1}^b + b\xi_j^a \xi_{j+1}^{b-1}\big)
+ \sum_{l=j+2, l \equiv 1 \pmod{2}}^{k-1} \big(a\xi_l^{a-1} \xi_{l+1}^b + b\xi_l^a \xi_{l+1}^{b-1}\big) \Bigg].
\]

This simplifies to:
\[
f_{k,a,b}(\vec{x}_i) - f_{k,a,b}(\vec{x}_{i+1}) \geq r^{n+1} \Bigg[\sum_{\substack{i < j \\ i \equiv 1 \!\!\!\! \pmod{2}}} \big(au_i^{a-1}u_{i+1}^b + bu_i^a u_{i+1}^{b-1}\big)
\]
\[
- \big(a(u_j + 2r^{n+1})^{a-1}(u_{j+1} + 2r^{n+1})^b + b(u_j + 2r^{n+1})^a(u_{j+1} + 2r^{n+1})^{b-1}\big)
\]
\[
+ \sum_{\substack{i > j \\ i \equiv 1 \!\!\!\! \pmod{2}}} \big(a(u_i + r^{n+1})^{a-1}(u_{i+1} + r^{n+1})^b + b(u_i + r^{n+1})^a(u_{i+1} + r^{n+1})^{b-1}\big) \Bigg]
\]

\[
\geq \sum_{\substack{i \neq j \\ i \equiv 1 \!\!\!\! \pmod{2}}} u_i^{a-1}u_{i+1}^{b-1}(b u_i + a u_{i+1}) 
- (u_j + r^n - r^{n+1})^{a-1}(u_{j+1} + r^n - r^{n+1})^{b-1}
\]
\[
\cdot \big[b(u_j + r^n - r^{n+1}) + a(u_{j+1} + r^n - r^{n+1})\big].
\]

Note that since \(a, b \geq 1\) and all \(u_i > 0\), the function
\[
(u_i, u_{i+1}) \mapsto (u_i+ r^n - r^{n+1})^{a-1} (u_{i+1}+ r^n - r^{n+1})^{b-1} [b (u_i+ r^n - r^{n+1}) + a (u_{i+1}+ r^n - r^{n+1})]
\]
is strictly increasing in both arguments. Therefore, the pair \((u_M, u_{M+1})\), which maximizes \(u_i^a u_{i+1}^b\), also maximizes the above expression. Hence, we obtain:

\[
\geq \sum_{\substack{i \neq M \\ i \equiv 1 \!\!\!\! \pmod{2}}} u_i^{a-1}u_{i+1}^{b-1}(b u_i + a u_{i+1}) 
- (u_M + r^n - r^{n+1})^{a-1}(u_{M+1} + r^n - r^{n+1})^{b-1}
\]
\[
\cdot \big[b(u_M + r^n - r^{n+1}) + a(u_{M+1} + r^n - r^{n+1})\big] \geq 0.
\]

Hence, we have shown that for any \(\omega_i \in S_2^k\), \(\bigcup_{\omega \leq \omega_i} Q_{\omega}\) is an interval. Furthermore, for any \(\omega_i \in \{0,1\}^k \setminus S_2^k\), it is easy to see that:  
\[
\bigcup_{\omega \leq \omega_i} Q_{\omega} \subseteq \bigcup_{\omega \leq \omega_{k/2+1}} Q_{\omega},
\]  
which is precisely \(Q_{\omega_{k/2+1}}\). Thus, the proof is complete.  
\end{proof}

\begin{lem}\label{lem}
Suppose \(k \in \mathbb{N}\), \(k \geq 2\), and \(u_1, u_2, \ldots, u_k \in L_n\). If the following inequality holds:  
\[
\sum_{i \neq M, i \equiv 1 \pmod{2}} u_i^{a-1} u_{i+1}^{b-1} (b u_i + a u_{i+1}) \geq (u_M + r^n)^{a-1} (u_{M+1} + r^n)^{b-1} 
\]
\[
\cdot \, [b (u_M + r^n) + a (u_{M+1} + r^n)],
\]  
where \(u_M\) and \(u_{M+1}\) satisfy  
\[
u_M^a u_{M+1}^b = \max\{u_1^a u_2^b, u_3^a u_4^b, \ldots, u_{k-1}^a u_k^b\},
\]  
then the following inclusion holds:  
\[
f_{k,a,b}(I_{u_1} \times I_{u_2} \times \cdots \times I_{u_k}) \subseteq f_{k,a,b}(C_\alpha^k).
\]  

Moreover, we have the precise equality:  
\[
f_{k,a,b}(I_{u_1} \times I_{u_2} \times \cdots \times I_{u_k}) = f_{k,a,b}((I_{u_1} \times I_{u_2} \times \cdots \times I_{u_k}) \cap C_\alpha^k).
\]
\end{lem}
\begin{proof}
Recall from Definition \ref{defn} that \(F_n\) denotes the collection of level-\(n\) intervals of the form \(I_u = [u, u + r^n]\). For each \(i = 1, 2, \ldots, k\) and integer \(l \geq n\), define
\[
F_{i,l} := \left\{ I \in F_l : I \subseteq I_{u_i} \right\},
\quad \text{and} \quad 
C_{i,l} := \bigcup_{A \in F_{i,l}} A \subseteq I_{u_i}.
\]
Then, for all \( l \geq n \), we have the inclusion
\[
f_{k,a,b}(I_{u_1} \times I_{u_2} \times \cdots \times I_{u_k}) 
\supseteq 
f_{k,a,b}(C_{1,l} \times C_{2,l} \times \cdots \times C_{k,l}).
\]

By Lemma \ref{linear}, it follows that:  
\[
f_{k,a,b}(C_\alpha^k) = \bigcap_{l=1}^\infty f_{k,a,b}(C_l^k) = \bigcap_{l=n}^\infty f_{k,a,b}(C_l^k).
\]  
Since \(C_{i,l} \subseteq C_l\) for \(i = 1, 2, \ldots, k\), we have:  
\[
\bigcap_{l=n}^\infty f_{k,a,b}(C_{1,l} \times C_{2,l} \times \cdots \times C_{k,l}) \subseteq \bigcap_{l=n}^\infty f_{k,a,b}(C_l^k) = f_{k,a,b}(C_{\alpha}^k).
\]

Thus, it suffices to show:  
\[
f_{k,a,b}(I_{u_1} \times I_{u_2} \times \cdots \times I_{u_k}) \subseteq f_{k,a,b}(C_{1,l} \times C_{2,l} \times \cdots \times C_{k,l}) \quad \text{for all } l \geq n.
\]

We will prove this by induction on \(l\).

When \(l = n\), \(C_{i,n} = I_{u_i}\) for \(i = 1, 2, \ldots, k\). Thus, the statement holds trivially.

When $l = n + 1$, $C_{i,n + 1} = I_{u_i,0} \cup I_{u_i,1}$ for $i = 1, 2, \ldots, k$. Thus, the statement is equivalent to the Lemma \ref{fundlem}.
  
Assume the statement holds for some \(l_k \geq n\), i.e.,  
\[
f_{k,a,b}(I_{u_1} \times I_{u_2} \times \cdots \times I_{u_k}) \subseteq f_{k,a,b}(C_{1,l_k} \times C_{2,l_k} \times \cdots \times C_{k,l_k}).
\]

We now show that the statement holds for \(l = l_k + 1\), i.e.,  
\[
f_{k,a,b}(I_{u_1} \times I_{u_2} \times \cdots \times I_{u_k}) \subseteq f_{k,a,b}(C_{1,l_k+1} \times C_{2,l_k+1} \times \cdots \times C_{k,l_k+1}).
\]

Let \(y \in f_{k,a,b}(I_{u_1} \times I_{u_2} \times \cdots \times I_{u_k})\). Then there exist \(u_1', u_2', \ldots, u_k' \in L_{l_k}\) such that \(I_{u_i'} \in F_{i,l_k}\) for \(i = 1, 2, \ldots, k\), and \(y \in f_{k,a,b}(I_{u_1'} \times I_{u_2'} \times \cdots \times I_{u_k'})\).  

Let \(u_M'^a u_{M+1}'^b = \max\{u_1'^a u_2'^b, u_3'^a u_4'^b, \ldots, u_{k-1}'^a u_k'^b\}\). Then we have:  
\[
u_M' \leq u_M + r^n - r^{l_k}, \quad u_{M+1}' \leq u_{M+1} + r^n - r^{l_k}, \quad u_i' \geq u_i \quad \text{for } i = 1, 2, \ldots, k.
\]

Using the conditions of the lemma, we derive:  
\[
\sum_{i \neq M', i \equiv 1 \pmod{2}} u_i'^{a-1} u_{i+1}'^{b-1} (b u_i' + a u_{i+1}') \geq \sum_{i \neq M, i \equiv 1 \pmod{2}} u_i^{a-1} u_{i+1}^{b-1} (b u_i + a u_{i+1})  
\]
\[
\geq (u_M + r^n)^{a-1}(u_{M+1} + r^n)^{b-1} [b(u_M + r^n) + a(u_{M+1} + r^n)] \geq 
\]
\[
(u_M' + r^{l_k})^{a-1}(u_{M+1}' + r^{l_k})^{b-1} [b(u_M' + r^{l_k}) + a(u_{M+1}' + r^{l_k})] \geq 
\]
\[
(u_M' + r^{l_k} - r^{l_k+1})^{a-1}(u_{M+1}' + r^{l_k} - r^{l_k+1})^{b-1} [b(u_M' + r^{l_k} - r^{l_k+1}) + a(u_{M+1}' + r^{l_k} - r^{l_k+1})].
\]

By Lemma \ref{fundlem}, we obtain:  
\[
f_{k,a,b}(I_{u_1'} \times I_{u_2'} \times \cdots \times I_{u_k'}) = f_{k,a,b}((I_{u_1',0} \cup I_{u_1',1}) \times (I_{u_2',0} \cup I_{u_2',1}) \times \cdots \times (I_{u_k',0} \cup I_{u_k',1})).
\]

Thus, there exists a \(v = (v_1, v_2, \ldots, v_k) \in \{0,1\}^k\) such that:  
\[
y \in f_{k,a,b}(I_{u_1',v_1} \times I_{u_2',v_2} \times \cdots \times I_{u_k',v_k}).
\]

Note that \(I_{u_i',v_i} \in F_{i,l_k+1}\) for \(i = 1, 2, \ldots, k\). Hence, \(y \in f_{k,a,b}(C_{1,l_k+1} \times C_{2,l_k+1} \times \cdots \times C_{k,l_k+1})\), and the statement holds for \(l_k + 1\).  

The proof is complete.  
\end{proof}

We now introduce some notation that will be used frequently in the following. First we define some commonly used constants

\[
n_* = n_*(r, a) = \lfloor -\log_r a \rfloor + 1,
\]

\[
k^{(1)} = k^{(1)}(r,a,b) = \left\lceil 2\frac{(1-r+r^{n_*})^{a-1}[b(1-r+r^{n_*})+a]}{(1-r)^{a-1}(1-r^{n_*})^{b-1}[b(1-r)+a(1-r^{n_*})]} + 2\right\rceil,
\]

\[
k_{\ast} = \max \left\{ k^{(1)}, \frac{2(1-r)}{r} \right\}.
\]

For simplicity in presentation, we list some conditions on $k$.

\begin{align*}
\text{(C1)} & \quad k \geq 2(\frac{1}{1-r})^{a+b-1} + 2, \\
\text{(C2)} & \quad k \geq \max \left\{ k_{\ast}, 2(\frac{1}{1-r})^{a+b-1} + 2 \right\}.\\
\end{align*}

\begin{lem}\label{C1}
For integer $k_1$ that satisfies (C1), $\left[ \frac{k_1}{2}(1-r)^s, \frac{k_1}{2} \right] \subseteq f_{k_1,a,b}(C^{k_1}_{\alpha})$.
\end{lem}
\begin{proof}
The proof of this lemma is mainly based on Lemma \ref{lem}. we will check the condition of Lemma \ref{lem}. For  $u_i = 1-r$ for $i = 1, 2, \ldots, k_1$ and $n = 1$.
\[
LHS = (\frac{k_1}{2}-1)((1-r)^{a-1}(1-r)^{b-1}(b(1-r)+a(1-r)) = (\frac{k_1}{2}-1)(1-r)^{a + b -1}(b + a)\geq
\]
\[
(\frac{1}{1-r})^{a+b-1}(1-r)^{a + b -1}(b + a) = (b+a) = 
\]
\[
((1-r) + r)^{a-1}((1-r) + r)^{b-1}[b((1-r) + r)+a((1-r) + r)] = RHS
\]

Based on these condition and Lemma \ref{lem}, we can easily induce that 
\[
\left[ \frac{k_1}{2}(1-r)^s, \frac{k_1}{2} \right] \subseteq f_{k_1,a,b}(C^{k_1}_{\alpha}).
\]
\end{proof}

\begin{lem}\label{C2}
If the integer \(k_2\) satisfies condition (C2), then we have:
\[
\left[ \frac{k_{\ast}}{2}(1-r)^a(1-r^{n_\ast})^b, \frac{k_2}{2} \right] \subseteq f_{k_2,a,b}(C^{k_{2}}_{\alpha}).
\]
\end{lem}

\begin{proof}
Using the fact that:
\[
k_{\ast} \geq k^{(1)} \geq 2\frac{(1-r+r^{n_\ast})^{a-1}[b(1-r+r^{n_\ast})+a]}{(1-r)^{a-1}(1-r^{n_\ast})^{b-1}[b(1-r)+a(1-r^{n_\ast})]} + 2,
\]
we obtain:
\[
\left(\frac{k_{\ast}}{2} - 1\right)(1-r)^{a-1}(1-r^{n_\ast})^{b-1}[b(1-r)+a(1-r^{n_\ast})] \geq (1-r+r^{n_\ast})^{a-1}[b(1-r+r^{n_\ast})+a].
\]

Following a similar method as in the previous proof, we check the conditions of Lemma \ref{lem}. Set \(u_i = 1-r\) for \(i = 1, 3, \ldots, k_2-1\), \(u_i = 1-r^{n_\ast}\) for \(i = 2, 4, \ldots, k_2\). Then:
\[
LHS = (\frac{k_2}{2}-1)(1-r)^{a-1}(1-r^{n_\ast})^{b-1}(b(1-r)+a(1-r^{n_\ast})) \geq
\]
\[
\frac{(1-r+r^{n_\ast})^{a-1}[b(1-r+r^{n_\ast})+a]}{(1-r)^{a-1}(1-r^{n_\ast})^{b-1}[b(1-r)+a(1-r^{n_\ast})]}(1-r)^{a-1}(1-r^{n_\ast})^{b-1}(b(1-r)+a(1-r^{n_\ast})) \geq
\]
\[
((1-r) + r^{n_\ast})^{a-1}((1-r^{n_\ast}) + r^{n_\ast})^{b-1}[b((1-r) + r^{n_\ast})+a((1-r^{n_\ast}) + r^{n_\ast})] = RHS.
\]

Based on these conditions and Lemma \ref{lem}, we can deduce:
\[
\left[ \frac{k_{\ast}}{2}(1-r)^a(1-r^{n_\ast})^b, \frac{k_{\ast}}{2}(1 - r + r^{n_\ast})^a \right] \subseteq f_{k_{\ast},a,b}(C^{k_{\ast}}_{\alpha}).
\]

Note that \(n_\ast \leq -\log_r a + 1\), so \(r^{n_\ast} \geq \frac{r}{a}\). Then:
\begin{align*}
\frac{k_{\ast}}{2} \left( (1-r + r^{n_{\ast}})^a - (1-r)^a(1-r^{n_{\ast}})^b \right) &\geq \frac{k_{\ast}}{2}(1-r)^a((1+\frac{r^{n_{\ast}}}{1-r})^a - (1-r^{n_{\ast}})^b) \\
&\geq \frac{k_{\ast}}{2}(1-r)^a((1+\frac{ar^{n_{\ast}}}{1-r}) - 1) \\
&\geq \frac{k_{\ast}}{2}(1-r)^a(\frac{r}{1-r}) \geq (1-r)^a \\
\end{align*}

Since \(0, 1 - r \in C_{\alpha}\) and \(k_2 \geq k_{\ast}\), using Lemma \ref{plus}, we derive:
\[
[\frac{k_{\ast}}{2}(1-r)^a(1-r^{n_\ast})^b, \frac{k_{\ast}}{2}(1 - r + r^{n_\ast})^a + \frac{(k_2- k_{\ast})}{2}(1-r)^a] \subseteq f_{k_{\ast}-k_{\ast}+k_2,a,b}(C_{\alpha}^{k_{\ast}-k_{\ast}+k_2}).
\]

Noting that \(k_2 \geq 2\left(\frac{1}{1-r}\right)^{s-1} + 2\), Based on Lemma \ref{C1}, we conclude:
\[
\left[ \frac{k_2}{2}(1-r)^s, \frac{k_2}{2} \right] \subseteq f_{k_2,a,b}(C^{k_2}_{\alpha}).
\]

Finally, since:
\[
\frac{1}{2}\left(k_{\ast}(1 - r + r^{n_\ast})^a + (k_2- k_{\ast})(1-r)^a\right) \geq \frac{k_2}{2}(1-r)^a \geq \frac{k_2}{2}(1-r)^s,
\]
combining the conditions and formulas above, we obtain:
\[
[\frac{k_{\ast}}{2}(1-r)^a(1-r^{n_\ast})^b, \frac{k_2}{2}] \subseteq f_{k_2,a,b}(C_{\alpha}^{k_2}),
\]
for \(k_2 \geq \max \{ k_{\ast}, 2(\frac{1}{1-r})^{s-1} + 2 \}\).
\end{proof}

\subsection{Complete Proof of Theorem \ref{FMT}}
\begin{proof}
It is not hard to find that $r \leq \frac{1}{2}$, thus we have \(k\) satisfies conditions (C1) and (C2), by the statement of the fundamental theorem. By applying Lemmas \ref{C1} and \ref{C2}, we can conclude that:  
\[
[\frac{k_{\ast}}{2}(1-r)^a(1-r^{n_\ast})^b, \frac{k}{2}] \subseteq f_{k,a,b}(C_{\alpha}^{k}).
\]
Noting the fact that:  
\[
k \geq \left\lceil2\frac{(1-r+r^{n_*})^{a-1}[b(1-r+r^{n_*})+a]}{(1-r)^{a-1}(1-r^{n_*})^{b-1}[b(1-r)+a(1-r^{n_*})]} + 2\right\rceil \cdot \left(\frac{1-r}{r}\right)^a,
\] 
And
\[
k \geq 2\left(\frac{1-r}{r}\right)^{a+1},
\] 
we can conclude that \(\frac{k}{2} r^a \geq \frac{k_{\ast}}{2}(1-r)^a \geq \frac{k_{\ast}}{2}(1-r)^a(1-r^{n_\ast})^b\).  

Combining this result with Corollary \ref{linear} and Lemma \ref{linearcomb}
, we deduce that:  
\[
[0, \frac{k}{2}] = \{0\} \cup \bigcup_{n=0}^\infty \left[\frac{k_{\ast}}{2}(1-r)^a (1-r^{n_\ast})^b r^{n a}, \frac{k}{2} r^{n a}\right] = f_{k,a,b}(C_{\alpha}^k),
\]  

Since the exponents \(a\) and \(b\) are symmetric under the condition \(a + b = s\), we can consider the condition by swapping \(a\) and \(b\). Specifically, we analyze the alternative condition:
\[
k^T \geq  \max \Bigg\{
\left\lceil2\frac{(1-r+r^{n_*})^{b-1}[a(1-r+r^{n_*})+b]}{(1-r)^{b-1}(1-r^{n_*})^{a-1}[a(1-r)+b(1-r^{n_*})]} + 2\right\rceil 
\cdot \left(\frac{1-r}{r}\right)^b,
\]
\[
2\left(\frac{1-r}{r}\right)^{b+1},
2\left(\frac{1}{1-r}\right)^{s-1} + 2
\Bigg\}.
\]

To analyze the behavior of \(k\), we divide its range into three components and treat each part as a function of \(a\). Define the following functions:
\[
E_1(a) = 
    \left\lceil 2\frac{(1-r+r^{n_*})^{a-1}[b(1-r+r^{n_*})+a]}{(1-r)^{a-1}(1-r^{n_*})^{b-1}[b(1-r)+a(1-r^{n_*})]} + 2
    \right\rceil \left(\frac{1 - r}{r}\right)^a,
\]
\[
E_2(a) = 2\left(\frac{1 - r}{r}\right)^{a+1},
\] 
\[
E_3 = 2 \left( \frac{1}{1 - r} \right)^{s-1} + 2.
\]

Each component corresponds to one of the terms in the range of \(k\). By analyzing these expressions, we can compare \(k\) for different values of \(a\). We note that \(E_2(a)\) is a strictly increasing function of \(a\), while \(E_3\) is independent of \(a\). It is not difficult to notice the following lemma:
\begin{lem}\label{E1}
For fixed \( s = a + b \) and \( r \in (0, \frac{1}{2}) \), the expression
\(
E_1(a)\) is a strictly increasing function of \( a \) for \( a = 1, 2, \ldots, s - 1 \).
\end{lem}
\begin{proof}
Define
\[
F(a) = \frac{(1 - r + r^{n_*})^{a-1} \left(b(1 - r + r^{n_*}) + a\right)}{(1 - r)^{a-1} (1 - r^{n_*})^{b-1} \left(b(1 - r) + a(1 - r^{n_*})\right)},
\]
where \( b = s - a \).

Consider the ratio
\[
\frac{E_1(a+1)}{E_1(a)} = \frac{2\lceil F(a+1) \rceil + 2}{2\lceil F(a) \rceil + 2} \times \frac{1 - r}{r}.
\]
Since \( F(a) \) increases with \( a \) (as \( n_* \) increases with \( a \) and each component of \( F(a) \) is either increasing or remains constant), we have \( \lceil F(a+1) \rceil \geq \lceil F(a) \rceil \).

Thus,
\[
\frac{E_1(a+1)}{E_1(a)} \geq \frac{2\lceil F(a) \rceil + 2}{2\lceil F(a) \rceil + 2} \times \frac{1 - r}{r} = \frac{1 - r}{r} > 1.
\]
Therefore, \( E_1(a+1) > E_1(a) \) for all \( a = 1, 2, \ldots, s - 1 \).

\end{proof}

Since we have $a$ $\leq$ $b$ in the statement of theorem, we can conclude that $k \leq k^T$. By selecting $k$ corresponding to the original condition $a$ $\leq$ $b$, we ensure the minimal value of $k$, which satisfies the desired constraints. This completes the proof.

\end{proof}
For the specific case of the Cantor ternary set \(C\), we derive the following corollary. This result follows directly from Theorem \ref{FMT} by setting \(r = \frac{1}{3}\) and leveraging the inequality \( \frac{1}{a} \geq r^{n_\ast} \geq \frac{r}{a}\) to simplify the bounds.
  
\begin{cor}\label{cor}  
Suppose that $\alpha > 1$, let \(a, b \in \mathbb{Z^+}\) and \(s = a + b\) be an integer with \(s \geq 2\). Then, for any positive even integer \(k\) satisfying:  
\[
k \geq
\left\lceil  \frac{\left(1 + \frac{3}{2a}\right)^{a-1}}{\left(1 - \frac{1}{a}\right)^{b-1}} \left(1 + \frac{1 + \frac{b}{a}}{b \cdot \frac{2}{3} + a - 1}\right) + 1\right\rceil \cdot 2^{a+1},
\]
the following holds:
\[
[0, \frac{k}{2}] = \left\{x_1^a x_2^b + x_3^a x_4^b + \cdots + x_{k-1}^a x_k^b : x_1, x_2, \ldots, x_k \in C \right\}.
\]
\end{cor}

Based on the proof above we can generalize the original theorem to address the scenario where \(a\) and \(b\) are not fixed but vary subject to \(a + b = s\). This ensures that the condition on \(k\) remains valid for any choice of \(a\) and \(b\) satisfying this constraint. The following corollary presents the result:

\begin{cor}\label{SMT}  
Suppose that $\alpha > 1$, let \(a, b \in \mathbb{Z^+}\) and \(s = a + b\) be an integer with \(s \geq 2\). Let  
\[
n'_* = n_*(r, s-1) = \lfloor -\log_r (s-1) \rfloor + 1.
\]  
Then, for any positive even integer \(k\) satisfying:  
\[
k \geq \max \Bigg\{
\left\lceil
2\frac{(1-r+r^{n'_*})^{s-2}[(1-r+r^{n'_*})+(s-1))]}{(1-r)^{s-2}[(1-r)+(s-1)(1-r^{n'_*})]} + 2
\right\rceil\left(\frac{1-r}{r}\right)^{s-1},
\]
\[
2\left(\frac{1-r}{r}\right)^s,
2\left(\frac{1}{1-r}\right)^{s-1} + 2
\Bigg\},
\]
the following holds:  
\[
[0, \frac{k}{2}] = \left\{x_1^a x_2^b + x_3^a x_4^b + \cdots + x_{k-1}^a x_k^b : x_1, x_2, \ldots, x_k \in C_\alpha \right\}.
\]

\end{cor}

\section{Proof of the Main Theorem}\label{sec4}
In this section, we will give the complete proof of Theorem \ref{TMT}. In section \ref{sec3}, all terms shared the same exponents \(a\) and \(b\), satisfying \(a + b = s\). In this generalization, we allow the exponents \(a_i\) for each term \(x_i\) to be chosen freely, with the constraint that \(a_{2i-1} + a_{2i} = s\) for all \(i\). However, to minimize the value of \(k\), all \(a_{2i-1}\) and \(a_{2i}\) are taken symmetrically around \(s/2\), ensuring consistency across all terms. This uniform distribution also ensures optimal scaling behavior in the parameter \(n''_*\) and maintains consistency with the structure of the Cantor set.

Recall that \(
r = \frac{1}{2} \left(1 - \frac{1}{\alpha} \right)\) . The following theorem is the full version of Theorem \ref{MResult}.

\begin{thm}\label{TMT}  Suppose that $\alpha > 1$,
let \(a_i > 0\) be integers satisfying:  
\[
a_{2i-1} + a_{2i} = s, \quad i = 1, 2, \dots, \frac{k}{2},
\]  
where \(s \geq 2\) is a given constant. Then, for any positive even integer \(k\) satisfying:  
\[
k \geq \max \Bigg\{
\left\lceil 2\frac{(1-r+r^{n''_*})^{\overline{a}-1} \big[\overline{b}(1-r+r^{n''_*}) + \overline{a}\big]}{(1-r)^{\overline{a}-1} (1-r^{n''_*})^{\overline{b}-1} \big[\overline{b}(1-r) + \overline{a}(1-r^{n''_*})\big]} + 2
\right\rceil \cdot \left(\frac{1-r}{r}\right)^{\overline{a}}
\]
\[
2\left(\frac{1-r}{r}\right)^{\overline{a}+1}, \quad 2\left(\frac{1}{1-r}\right)^{s-1} + 2
\Bigg\},
\]  
the following holds:  
\[
[0, \frac{k}{2}] = \left\{ \sum_{i=1}^{k/2} x_{2i-1}^{a_{2i-1}} x_{2i}^{a_{2i}} \;\middle|\; x_1, x_2, \dots, x_k \in C_\alpha \right\}.
\]

Here, \( \overline{a} \) and \( \overline{b} \) are defined based on the parity of \(s\) as follows:  
\[
\overline{a} = 
\begin{cases} 
\frac{s - 1}{2}, & \text{if } s \text{ is odd,} \\
\frac{s}{2}, & \text{if } s \text{ is even,}
\end{cases}
\quad \overline{b} = s - \overline{a}.
\]

Additionally, let \(n''_*\) be defined as:  
\[
n''_* = n_*(r, \overline{a}) = \lfloor -\log_r \overline{a} \rfloor + 1.
\]  
\end{thm}

\subsection{Preparatory Lemmas}
Carefully examining the proofs of Lemma \ref{fundlem} and \ref{lem}, we find that the proof still holds when we scramble the values of $a$ and $b$ into $a_i$, so we can naturally draw the following corollary.
\begin{cor}\label{fundcor}
Suppose \(k \in \mathbb{N}\), \(k \geq 2\), and \(u_1, u_2, \ldots, u_k \in L_n\). If the following inequality holds:  
\[
\sum_{i=1}^{k/2} u_{2i-1}^{a_{2i-1}-1} u_{2i}^{a_{2i}-1}(a_{2i} u_{2i-1} + a_{2i-1} u_{2i}) \geq (u_M + r^n - r^{n+1})^{a_M-1} (u_{M+1} + r^n - r^{n+1})^{a_{M+1}-1}
\]
\[
\cdot \, [a_{M+1} (u_M + r^n - r^{n+1}) + a_M (u_{M+1} + r^n - r^{n+1})],
\]  
where \(u_M\) and \(u_{M+1}\) satisfy  
\[
u_M^{a_M} u_{M+1}^{a_{M+1}} = \max\{u_1^{a_1} u_2^{a_2}, u_3^{a_3} u_4^{a_4}, \ldots, u_{k-1}^{a_{k-1}} u_k^{a_k}\},
\]  
then the following equality holds:  
\[
f_{k,\vec{a}}(I_{u_1} \times I_{u_2} \times \cdots \times I_{u_k}) = f_{k,\vec{a}}[(I_{u_1,0} \cup I_{u_1,1}) \times (I_{u_2,0} \cup I_{u_2,1}) \times \cdots \times (I_{u_k,0} \cup I_{u_k,1})],
\]  
where \(I_{u_k}\) and \(I_{u_k,\sigma}\) are defined as in Equations (\ref{leftp1}) and (\ref{leftp2}), respectively.
\end{cor}
\begin{cor}\label{cor}
Suppose \(k \in \mathbb{N}\), \(k \geq 2\), and \(u_1, u_2, \ldots, u_k \in L_n\). If the following inequality holds:  
\[
\sum_{i=1}^{k/2} u_{2i-1}^{a_{2i-1}-1} u_{2i}^{a_{2i}-1}(a_{2i} u_{2i-1} + a_{2i-1} u_{2i}) \geq (u_M + r^n)^{a_M-1} (u_{M+1} + r^n)^{a_{M+1}-1}
\]
\[
\cdot \, [a_{M+1} (u_M + r^n) + a_M (u_{M+1} + r^n)],
\]  
where \(u_M\) and \(u_{M+1}\) satisfy  
\[
u_M^{a_M} u_{M+1}^{a_{M+1}} = \max\{u_1^{a_1} u_2^{a_2}, u_3^{a_3} u_4^{a_4}, \ldots, u_{k-1}^{a_{k-1}} u_k^{a_k}\},
\]  
then the following inclusion holds:  
\[
f_{k,s}(I_{u_1} \times I_{u_2} \times \cdots \times I_{u_k}) \subseteq f_{k,\vec{a}}(C_{\alpha}^k).
\]  

Moreover, we have the precise equality:  
\[
f_{k,\vec{a}}(I_{u_1} \times I_{u_2} \times \cdots \times I_{u_k}) = f_{k,\vec{a}}((I_{u_1} \times I_{u_2} \times \cdots \times I_{u_k}) \cap C_{\alpha}^k).
\]
\end{cor}
Using these two corollaries, we can similarly draw the corresponding conclusions. Like the proof of fundamental theorem, we now introduce some notation that will be used frequently in the following. First we define some commonly used constants

\[
n''_* = n_*(r, \overline{a}) = \lfloor -\log_r \overline{a} \rfloor + 1,
\] 

\[
k'^{(1)} = \left\lceil2\frac{(1-r+r^{n''_\ast})^{\overline{a}-1}[(s-\overline{a})(1-r+r^{n''_\ast})+\overline{a}]}{(1-r)^{\overline{a}-1}(1-r^{n''_\ast})^{(s-\overline{a})-1}[(s-\overline{a})(1-r)+\overline{a}(1-r^{n''_\ast})]} + 2\right\rceil,
\]

\[
k'_{\ast} = \max \left\{ k'^{(1)}, \frac{2(1-r)}{r} \right\}.
\]

For simplicity in presentation, we list some conditions on $k$.

\begin{align*}
\text{(C3)} & \quad k \geq \max \left\{ k'_{\ast}, 2(\frac{1}{1-r})^{s-1} + 2 \right\}.\\
\end{align*}

\begin{lem}\label{3lem}
For an integer \(k_1\) that satisfies condition (C1), the following inclusion holds:  
\[
\left[ \frac{k_1}{2}(1-r)^s, \frac{k_1}{2} \right] \subseteq f_{k_1,\vec{a}}(C^{k_1}_{\alpha}).
\]
\end{lem}

\begin{proof}
The proof is based on Lemma \ref{cor}. We check the condition of Lemma \ref{cor} for \(u_i = 1-r\) for \(i = 1, 2, \ldots, k_1\), and \(n = 1\):
\[
LHS = \sum_{i=1,i \neq M}^{k/2} (1-r)^{a_{2i-1}-1} (1-r)^{a_{2i}-1}(a_{2i} (1-r) + a_{2i-1} (1-r)) = (\frac{k_1}{2}-1)(1-r)^{s-1}(s) 
\]
\[
\geq \left(\frac{1}{1-r}\right)^{s-1}(1-r)^{s-1}(s) = s =
\]
\[
((1-r) + r)^{a_{M}-1}((1-r) + r)^{a_{M+1}-1}[a_{M+1}((1-r) + r) + a_{M}((1-r) + r)] = RHS.
\]

Based on these conditions and Lemma \ref{cor}, we deduce:  
\[
\left[ \frac{k_1}{2}(1-r)^s, \frac{k_1}{2} \right] \subseteq f_{k_1,\vec{a}}(C^{k_1}_{\alpha}).
\]
\end{proof}

\begin{lem}\label{4lem}
If the integer \(k_2\) satisfies condition (C3), then the following holds:  
\[
\left[ \frac{k'_{\ast}}{2}(1-r)^{\overline{a}}(1-r^{n''_\ast})^{\overline{b}}, \frac{k_2}{2} \right] \subseteq f_{k_2,\vec{a}}(C^{k_2}_{\alpha}).
\]
\end{lem}

\begin{proof}
Using the fact that:
\[
k'_{\ast} \geq k'^{(1)} \geq 2\frac{(1-r+r^{n''_\ast})^{\overline{a}-1}[\overline{b}(1-r+r^{n''_\ast})+\overline{a}]}{(1-r)^{\overline{a}-1}(1-r^{n''_\ast})^{\overline{b}-1}[\overline{b}(1-r)+\overline{a}(1-r^{n''_\ast})]} + 2,
\]
we obtain:
\[
\left(\frac{k'_{\ast}}{2} - 1\right)(1-r)^{\overline{a}-1}(1-r^{n''_*})^{\overline{b}-1}(\overline{b} (1-r) + \overline{a} (1-r^{n''_*}))
\]
\[
\geq (1-r+r^{n''_\ast})^{\overline{a}-1}[\overline{b}(1-r+r^{n''_\ast})+\overline{a}].
\]

Following a similar method as in the previous proof, we check the conditions of Lemma \ref{cor}. Set \(u_i = 1-r\) for \(i = 1, 3, \ldots, k_2-1\), and \(u_i = 1-r^{n''_\ast}\) for \(i = 2, 4, \ldots, k_2\). Then:
\[
LHS = \sum_{i=1,i \neq M}^{k/2} (1-r)^{\overline{a}-1} (1-r^{n''_*})^{\overline{b}-1}(\overline{b} (1-r) + \overline{a} (1-r^{n''_*}))  
\]
\[
\geq \left(\frac{k_{2}}{2} - 1\right)(1-r)^{\overline{a}-1}(1-r^{n''_*})^{\overline{b}-1}(\overline{b} (1-r) + \overline{a} (1-r^{n''_*}))
\]
\[
\geq (1-r+r^{n''_\ast})^{\overline{a}-1}[\overline{b}(1-r+r^{n''_\ast})+\overline{a}] = RHS.
\]

Based on these conditions and Lemma \ref{cor}, we can deduce:
\[
\left[ \frac{k'_{\ast}}{2}(1-r)^{\overline{a}}(1-r^{n''_\ast})^{\overline{b}}, \frac{k'_{\ast}}{2}(1 - r + r^{n''_\ast})^{\overline{a}} \right] \subseteq f_{k'_{\ast},\vec{a}}(C^{k'_{\ast}}_{\alpha}).
\]

Note that \(n''_\ast \leq -\log_r \overline{a} + 1\), so \(r^{n''_\ast} \geq \frac{r}{\overline{a}}\). Then:
\begin{align*}
\frac{k'_{\ast}}{2} \left( (1-r + r^{n''_{\ast}})^{\overline{a}} - (1-r)^{\overline{a}}(1-r^{n''_\ast})^{\overline{b}} \right) &\geq \frac{k'_{\ast}}{2}(1-r)^{\overline{a}}((1+\frac{r^{n''_{\ast}}}{1-r})^{\overline{a}} - (1-r^{n''_{\ast}})^{\overline{b}}) \\
&\geq \frac{k'_{\ast}}{2}(1-r)^{\overline{a}}((1+\frac{\overline{a} r^{n''_{\ast}}}{1-r}) - 1) \\
&\geq \frac{k'_{\ast}}{2}(1-r)^{\overline{a}}(\frac{r}{1-r}) \geq (1-r)^{\overline{a}}.
\end{align*}

Since \(0, 1 - r \in C_{\alpha}\) and \(k_2 \geq k'_{\ast}\), using Lemma \ref{plus_fks}, we derive:
\[
[\frac{k'_{\ast}}{2}(1-r)^{\overline{a}}(1-r^{n''_\ast})^{\overline{b}}, \frac{k'_{\ast}}{2}(1 - r + r^{n''_\ast})^{\overline{a}} + \frac{(k_2- k'_{\ast})}{2}(1-r)^{\overline{a}}] \subseteq f_{k'_{\ast}-k'_{\ast}+k_2,\vec{a}}(C_{\alpha}^{k'_{\ast}-k'_{\ast}+k_2}).
\]

Noting that \(k_2 \geq 2\left(\frac{1}{1-r}\right)^{s-1} + 2\), based on Lemma \ref{3lem}, we conclude:
\[
\left[ \frac{k_2}{2}(1-r)^s, \frac{k_2}{2} \right] \subseteq f_{k_2,\vec{a}}(C^{k_2}_{\alpha}).
\]

Finally, since:
\[
\frac{1}{2}\left(k'_{\ast}(1 - r + r^{n''_\ast})^{\overline{a}} + (k_2- k'_{\ast})(1-r)^{\overline{a}}\right) \geq \frac{k_2}{2}(1-r)^{\overline{a}} \geq \frac{k_2}{2}(1-r)^s,
\]
combining the conditions and formulas above, we obtain:
\[
[\frac{k'_{\ast}}{2}(1-r)^{\overline{a}}(1-r^{n''_\ast})^{\overline{b}}, \frac{k_2}{2}] \subseteq f_{k_2,\vec{a}}(C_{\alpha}^{k_2}).
\]

for \(k_2 \geq \max \{ k'_{\ast}, 2(\frac{1}{1-r})^{s-1} + 2 \}\).
\end{proof}

\subsection{Complete Proof of Theorem \ref{TMT}}

\begin{proof}

We have already established that \(k\) satisfies conditions (C3), as stated in the main theorem. By applying Lemmas \ref{3lem} and \ref{4lem}, we can conclude that:  
\[
[\frac{k'_{\ast}}{2}(1-r)^{\overline{a}}(1-r^{n''_\ast})^{s-\overline{a}}, \frac{k}{2}] \subseteq f_{k,\vec{a}}(C_{\alpha}^k).
\]  
Noting the fact that:  
\[
k \geq \left\lceil 2\frac{(1-r+r^{n''_*})^{\overline{a}-1} \big[(s-\overline{a})(1-r+r^{n''_*}) + \overline{a}\big]}{(1-r)^{\overline{a}-1} (1-r^{n''_*})^{(s-\overline{a})-1} \big[(s-\overline{a})(1-r) + \overline{a}(1-r^{n''_*})\big]} + 2
\right\rceil \cdot \left(\frac{1-r}{r}\right)^{\overline{a}},
\]  
and
\[
k \geq 2\left(\frac{1-r}{r}\right)^{\overline{a}+1},
\]  
we can conclude that:  
\[
\frac{k}{2} r^{\overline{a}} \geq  \frac{k'_{\ast}}{2}(1-r)^{\overline{a}} \geq \frac{k'_{\ast}}{2}(1-r)^{\overline{a}}(1-r^{n''_\ast})^{s-\overline{a}}.
\]  

Combining this result with Corollary \ref{linear} and Lemma \ref{linearcomb-fks}, we deduce that:  
\[
[0, \frac{k}{2}] = \{0\} \cup \bigcup_{n=0}^\infty \left[\frac{k'_{\ast}}{2}(1-r)^{\overline{a}}(1-r^{n''_\ast})^{s-\overline{a}} r^{n \overline{a}}, \frac{k}{2} r^{n \overline{a}}\right] = f_{k,\vec{a}}(C_{\alpha}^k),
\]  
which is precisely the desired conclusion.  
\end{proof}  

For the specific case of the Cantor ternary set \(C\), we derive the following corollary, which was already stated in Theorem \ref{MResult}. This result follows directly from Theorem \ref{TMT} by setting \(r = \frac{1}{3}\) and leveraging the inequality \( \frac{1}{\overline{a}} \geq r^{n''_\ast} \geq \frac{r}{\overline{a}}\) to simplify the bounds.

\begin{cor}\label{final}
Suppose that $\alpha > 1$, let \(a_i > 0\) (\(i = 1, 2, \ldots, k\)) be integers satisfying:  
\[
a_{2i-1} + a_{2i} = s, \quad i = 1, 2, \ldots, \frac{k}{2},
\]   
where \(s > 2\) is a given constant. Then, for any positive even integer \(k\) satisfying:  
\[
k \geq 2^{\frac{s}{2}+1} \cdot \left\lceil \left(\frac{s+3}{s-2}\right)^{\frac{s}{2} - 1} \left( \frac{5 s + 6 }{5 s - 6 }\right) + 1 \right\rceil,
\]
the following holds:  
\[
[0, \frac{k}{2}] = \left\{ \sum_{i=1}^{k/2} x_{2i-1}^{a_{2i-1}} x_{2i}^{a_{2i}} \;\middle|\; x_1, x_2, \dots, x_k \in C \right\}.
\]

\end{cor}

It is worth noting that when  $s \leq 6$, a better bound can be implied directly by (\ref{4}) in \cite{Cui}. In contrast, when $s \geq 7$, the result of Corollary \ref{final} is nontrivial.

Based on the understanding of the above theorem, we can derive a more flexible corollary:

\begin{cor}\label{finalcor}  
Suppose that $\alpha > 1$, let \(a_i > 0\) be integers satisfying:  
\[
a_{2i-1} + a_{2i} = s, \quad i = 1, 2, \dots, \frac{k}{2},
\]  
where \(s \geq 2\) is a given constant. For each pair \((a_{2i-1}, a_{2i})\), let \(a_{\max}\) denote the maximum value of \(a_{2i-1}\), and \(b_{\min}\) denote the corresponding minimum value of \(a_{2i}\). Then, for any positive even integer \(k\) satisfying: 
\[
k \geq \max \Bigg\{
\left\lceil 2\frac{(1-r+r^{n'''_*})^{a_{\max}-1} \big[b_{\min}(1-r+r^{n'''_*}) + a_{\max}\big]}{(1-r)^{a_{\max}-1} (1-r^{n'''_*})^{b_{\min}-1} \big[b_{\min}(1-r) + a_{\max}(1-r^{n'''_*})\big]} + 2
\right\rceil \cdot \left(\frac{1-r}{r}\right)^{a_{\max}}
\]
\[
2\left(\frac{1-r}{r}\right)^{a_{\max}+1}, \quad 2\left(\frac{1}{1-r}\right)^{s-1} + 2
\Bigg\},
\]
the following holds: 
\[
[0, \frac{k}{2}] = \left\{ \sum_{i=1}^{k/2} x_{2i-1}^{a_{2i-1}} x_{2i}^{a_{2i}} \;\middle|\; x_1, x_2, \dots, x_k \in C_\alpha \right\}.
\]

Here, \(a_{\max} = \max_{i} a_{2i-1}\) and \(b_{\min} = \min_{i} a_{2i}\) represent the extreme values among all pairs. Additionally, let \(n'''_*\) be defined as:  
\[
n'''_* = n_*(r, a_{\max}) = \lfloor -\log_r a_{\max} \rfloor + 1.
\]  
\end{cor}  

\section{More Discussions}\label{sec5}

The cross terms in the non-diagonal form present considerable complexity. By examining specific low-dimensional cases and special forms, we can classify the different types of cross terms. This classification provides valuable insights into understanding the nature of these terms more effectively.

\subsection{The Cases for Lower Dimension}
\subsubsection{2 Dimension Case} 
When \(s = 2\), the only possible assignment of \(a\) and \(b\) is \(a = b = 1\). In this case, Jayadev S. Athreya, Bruce Reznick, and Jeremy T. Tyson provide the following conclusion regarding the Cantor ternary set\cite{Tyson}: the set of real numbers \(v\) that can be expressed in the form \(v = x y\), where \(x, y \in C\), is a closed subset of \([0, 1]\) with Lebesgue measure strictly between \(\frac{17}{21}\) and \(\frac{8}{9}\). Meanwhile, as mentioned in the introduction, Steinhaus proved in 1917 \cite{Stein} that the sum of two Cantor ternary sets \(C\) satisfies
\(
C + C = \{x + y : x, y \in C\} = [0, 2].
\)
This result allows us to extend compatibility downward to conclude that:
\[
\{x_1x_2 + x_3x_4 : x_1, x_2, x_3, x_4 \in C\} = [0, 2].
\]
Here, we observe that by setting \(x_2 = 1\) and \(x_4 = 1\), the problem reduces to the previously considered case, yielding the same result as above. The existence of such a result confirms that the minimal \(k\) for \(s = 2\) is \(k = 4\).

\subsubsection{3 Dimension Case}

When \(s = 3\), the allocation of \(a\) and \(b\) must satisfy \(a = 1\) and \(b = 2\) (or \(a = 2\), \(b = 1\), respectively), which, considering symmetry, is the only possible distribution. In this case, the work of \cite{Tyson} provides the following result for the Cantor ternary set: every \(u \in [0,1]\) can be expressed as \(u = x^2 y\) for some \(x, y \in C\). Based on this result, we can derive the following compatibility conclusion:  
\[
\{x_1 x_2^2 : x_1, x_2 \in C\} = [0, 1].
\]  
The existence of such a representation is thereby established, leading to the minimal value of \(k = 2\) for \(s = 3\).

For higher-dimensional cases, the allocation of \(a\) and \(b\) offers more than one possibility. To gain a thorough understanding of this, it is essential to study the specific form \(x y^b\).
\subsection{The Cases for the Form $ x y^b$}

Based on the results of \cite{Tyson}, we deduce the following corollary:

\begin{cor}  
Let \(b \geq 2\), and let \(I_u, I_v \in F_n\) be two level-\(n\) basic intervals in \(C_n\). The image of the function \(f_{1,1,b}(x, y) = x y^b\) satisfies:
\[
f_{1,1,b}(I_u \times I_v) = \bigcup_{\sigma_1, \sigma_2 \in \{0,1\}} f_{1,1,b}(I_{u,\sigma_1} \times I_{v,\sigma_2}),
\]
where \(I_{u,\sigma_1}\) and \(I_{v,\sigma_2}\) are the level-\((n+1)\) child intervals of \(I_u\) and \(I_v\), respectively. These subintervals are disjoint and collectively cover \(f_{1,1,b}(I_u \times I_v)\) without any gaps.  
\end{cor}  

By further applying the methods from their work, we obtain the following result regarding the Cantor ternary set: every \(u \in [0,1]\) can be expressed as \(u = x y^b\) for some \(x, y \in C\), where \(b \geq 2\). Consequently, the minimal \(k\) for this specific form \(x y^b\) is \(k = 2\).  

For the form \(x y^b\), a single pair \((x, y)\) is sufficient to cover \([0, 1]\). Furthermore, the range of numbers expressed as \(x y^b\) and \((x y^b)^d\) can be mapped bijectively. Extending this to a broader setting, for \(u \in (0, 1)\), the form \(x^a y^b\) satisfies \(u = x^a y^b\) if and only if \(u^{1/a} = x y^{b/a}\).  

In the work of \cite{Tyson}, Athreya, Reznick, and Tyson noted that for the Cantor ternary set, if \(u = x y^{b/a}\), \(b/a \geq 1\), and \(\left(\frac{2}{3}\right)^{1+{b/a}} > \frac{1}{3}\), then \(u\) fails to cover an interval in \([0, 1]\). This condition is satisfied when \(b/a < \frac{\log 2}{\log \frac{3}{2}}\). However, when \(b/a \geq \frac{\log 2}{\log \frac{3}{2}}\), the form \(x^a y^b\) becomes sufficient to fully cover \([0, 1]\).

When \(b/a < \frac{\log 2}{\log \frac{3}{2}}\), the situation becomes more intricate. Our theorem provides comprehensive conclusions for all possible cases, which highlights the significance of the result.

\section*{Acknowledgments} 
I am grateful for the ShanghaiTech IMS Tutor-Tutee system for undergraduate students. And I would like to thank Professor Boqing Xue for the guidance and indispensable discussions. I also appreciate the referees’ helpful comments and suggestions.

\end{document}